\documentclass[11pt,twoside]{article}
\usepackage{stmaryrd}
\usepackage{amssymb}
\usepackage{amsmath}
\usepackage{amsthm}
\usepackage{color}
\usepackage{mathrsfs}

\allowdisplaybreaks

\def\rr{{\mathbb R}}
\def\rn{{{\rr}^n}}

\def\supp{{\mathop\mathrm{\,supp\,}}}

\def\ls{\lesssim}
\def\gs{\gtrsim}

\def\beeqn{\begin{equation}}
\def\eneqn{\end{equation}}
\def\beeqns{\begin{equation*}}
\def\eneqns{\end{equation*}}
\def\beeqa{\begin{eqnarray}}
\def\eneqa{\begin{eqnarray}}
\def\beeqas{\begin{eqnarray*}}
\def\eneqas{\begin{eqnarray*}}
\def\besp{\begin{split}}
\def\ensp{\begin{split}}

\newtheorem{thm}{Theorem}[section]
\newtheorem{lem}[thm]{Lemma}
\newtheorem{rem}[thm]{Remark}
\newtheorem{cor}[thm]{Corollary}
\newtheorem{defn}[thm]{Definition}

\numberwithin{equation}{section}

\begin{document}

\title{\Large\bf  The boundedness of commutators of sublinear operators on Herz Triebel-Lizorkin spaces with variable
exponent}
\author{CHENGLONG FANG, YINGYING WEI AND JING ZHANG\footnote{Corresponding author}}

\date{}
\maketitle

{\noindent  {\bf Abstract:} In this paper,  the authors first discuss the characterization of Herz Triebel-Lizorkin spaces
with variable exponent via two families of operators. By this characterization, the authors prove that the Lipschitz commutators of sublinear operators is bounded from Herz spaces with variable exponent to Herz Triebel-Lizorkin spaces with variable exponent. As an application, the corresponding boundedness estimates for the commutators of maximal operator, Riesz potential operator and Calder\'{o}n-Zygmund operator are established.}

\bigskip
\bigskip


\medskip

\section{Introduction and statement of main result}
\quad

Given a locally integrable function $b$, the commutator $[b,T]$ is defined by
\begin{equation*}
[b,T]f=bTf-T(bf)
\end{equation*}
for suitable functions $f$. In 1995,  Paluszynski \cite{MP} proved that $b$ belongs to the Lipschitz spaces if and only if the commutator $[b,T]$ is bounded from Lebesgue spaces to Triebel-Lizorkin spaces, where $T$ is a Calder\'on-Zygmund  singular integral operator.  As we know, the Herz spaces is the Lebesgue spaces with power weights $|x|^\alpha$. In recent years, the Herz-type spaces have been paid more and more attention; see, for example, (\cite{FC2}-\cite{LY},\cite{SCJS}-\cite{XJ1}) and references therein. More recently, replacing Triebel-Lizorkin spaces by Herz Triebel-Lizorkin spaces in the results of Paluszynski,
Fang and Zhou \cite{FC2} obtained that $[b,T]$ is bounded from Herz spaces to Herz Triebel-Lizorkin spaces,
where $b$ belongs to the Lipschitz spaces and $T$ is a sublinear operators.
Inspired by \cite{FC2} and \cite{MP}, in this article,
the authors show the boundedness of Lipschitz commutators of sublinear operators on Herz Triebel-Lizorkin spaces with variable exponent.
To this ends,  the authors discuss the characterizations of the Herz Triebel-Lizorkin spaces with variable exponent.

To state our main results, we first recall some definitions and  notations.

\begin{defn}\label{def-Lip}
For $0<\beta<1$, the Lipschitz spaces $\dot{\Lambda}_{\beta}$ is the space of functions $f$ such that
$$\|f\|_{\dot{\Lambda}_{\beta}}=\sup_{x,h\in\mathbb{R}^{n},~h\neq0}\frac{|f(x+h)-f(x)|}{|h|^{\beta}}<\infty.$$
\end{defn}

\begin{defn}\label{D1}
Let $\Omega$ be a measurable set in $\mathbb{R}^{n}$ with $|\Omega|>0$. For a measurable function $p(\cdot):\Omega\rightarrow[1,\infty)$,
the variable exponent Lebesgue space $L^{p(\cdot)}(\Omega)$ is defined by
\begin{equation*}
L^{p(\cdot)}(\Omega):=\left\{f~is~ measurable: \rho_{p(\cdot)}(f/\lambda)<\infty ~for~some~constant~\lambda>0\right\},
\end{equation*}
where
\begin{equation*}
\rho_{p(\cdot)}(f):=\int_{\Omega}|f(x)|^{p(x)}dx.
\end{equation*}
The norm of $f$ in the variable exponent Lebesgue space $L^{p(\cdot)}(\Omega)$ is denoted by
\begin{equation*}
\|f\|_{L^{p(\cdot)}(\Omega)}:=\inf\left\{\lambda>0:\rho_{p(\cdot)}(f/\lambda)\leq1\right\}.
\end{equation*}
Then $L^{p(\cdot)}(\Omega)$ is a Banach space with the norm $\|\cdot\|_{L^{p(\cdot)}(\Omega)}$. Denote
\begin{equation*}
p^{-}:=ess\inf\{p(x):x\in\mathbb{R}^{n}\} \quad \text{and} \quad p^{+}:=ess\sup\{p(x):x\in\mathbb{R}^{n}\}.
\end{equation*}
Throughout the whole article, the set $\mathscr{P}(\mathbb{R}^{n})$ consists of all $p(\cdot)$ satisfying $p^{-}>1$~and~$p^{+}<\infty$.
\end{defn}

Given a locally integrable function $f$, the Hardy-Littlewood maximal operator $M$ is defined by
\begin{align}\label{def-M}
Mf(x):=\sup_{r>0}r^{-n}\int_{B(x,r)}|f(y)|dy \quad \text{for all} \  x\in\mathbb{R}^{n}.
\end{align}
Here and in the future,
\begin{equation*}
B(x,r):=\{y\in\mathbb{R}^{n}:|x-y|<r\}.
\end{equation*}
Let $\mathscr{B}(\mathbb{R}^{n})$ be the set of $p(\cdot)\in\mathscr{P}(\mathbb{R}^{n})$ satisfying the condition that $M$ is bounded on $L^{p(\cdot)}(\mathbb{R}^{n})$.
It is well known that if $p(\cdot)\in\mathscr{P}(\mathbb{R}^{n})$, then $p(\cdot)\in\mathscr{B}(\mathbb{R}^{n})$ is equivalent to $p'(\cdot)\in\mathscr{B}(\mathbb{R}^{n})$, where $1/p'(\cdot)=1-1/p(\cdot)$.
When $p(\cdot)\in\mathscr{B}(\mathbb{R}^{n})$, Izuki \cite{IM1} proved that there exist $\delta_{1} ,\delta_{2}\in(0,1)$ depending only on $p(\cdot)$ and $n$ such that for ball $B$ in $\mathbb{R}^{n}$ and all measurable subsets $S\subset B$,
\begin{equation}\label{1.1}
\frac{\|\chi_{S}\|_{L^{p(\cdot)}(\mathbb{R}^{n})}}{\|\chi_{B}\|_{L^{p(\cdot)}(\mathbb{R}^{n})}}\leq C\left(\frac{|S|}{|B|}\right)^{\delta_{1}}, \
\frac{\|\chi_{S}\|_{L^{p'(\cdot)}(\mathbb{R}^{n})}}{\|\chi_{B}\|_{L^{p'(\cdot)}(\mathbb{R}^{n})}}\leq C\left(\frac{|S|}{|B|}\right)^{\delta_{2}}.
\end{equation}

Now we recall the definitions of Herz spaces with variable exponents and Herz Triebel-Lizorkin spaces with variable exponents.
Set $B_{k}=B(0, 2^{k})=\{{x\in  \mathbb{R}^{n}:|x|<2^{k}}\}$, $E_{k}=B_{k}\backslash B_{k-1}$ and $\chi_{k}=\chi_{E_{k}}$ for $k\in\mathbb{Z}$.
For $m\in\mathbb{N}_{0}=\mathbb{N}\cup \{0\}$,  define
\begin{equation*}
\tilde{\chi}_{m}=
\begin{cases}
\chi_{E_{m}},\ &m\geq1,\\
\chi_{B_{0}},\ &m=0.
\end{cases}
\end{equation*}

\begin{defn}\label{D2}
Let $\alpha\in\mathbb{R}$, $0<q\leq\infty$ and $p(\cdot)\in\mathscr{P}(\mathbb{R}^{n})$.
\begin{enumerate}
\item[\rm (i)]
The homogeneous Herz spaces with variable exponents $\dot{K}^{\alpha,p}_{p(\cdot)}(\mathbb{R}^{n})$ is defined by
\begin{equation*}
\dot{K}^{\alpha,q}_{p(\cdot)}(\mathbb{R}^{n})=\left\{f\in L^{p(\cdot)}_{loc}(\mathbb{R}^{n}\backslash \{0\}):\|f\|_{\dot{K}^{\alpha,q}_{p(\cdot)}(\mathbb{R}^{n})}<\infty\right\},
\end{equation*}
where
\begin{equation*}
\|f\|_{\dot{K}^{\alpha,q}_{p(\cdot)}(\mathbb{R}^{n})}=\left\{\sum^{\infty}_{l=-\infty}(2^{\alpha l}\|f\chi_{l}\|_{L^{p(\cdot)}(\mathbb{R}^{n})})^{q}\right\}^{1/q}.
\end{equation*}

\item[\rm (ii)]
The non-homogeneous Herz spaces with variable exponents $K^{\alpha,p}_{p(\cdot)}(\mathbb{R}^{n})$ is defined by
\begin{equation*}
K^{\alpha,q}_{p(\cdot)}(\mathbb{R}^{n})=\left\{f\in L^{p(\cdot)}_{loc}(\mathbb{R}^{n}):\|f\|_{K^{\alpha,q}_{p(\cdot)}(\mathbb{R}^{n})}<\infty\right\},
\end{equation*}
where
\begin{equation*}
\|f\|_{K^{\alpha,q}_{p(\cdot)}(\mathbb{R}^{n})}=\left\{\sum^{\infty}_{m=0}(2^{\alpha m}\|f\tilde{\chi}_{m}\|_{L^{p(\cdot)}(\mathbb{R}^{n})})^{q}\right\}^{1/q}.
\end{equation*}
\end{enumerate}
\end{defn}

Let $\mathscr{S}(\mathbb{R}^{n})$ be the Schwartz space on $ \mathbb{R}^{n}$ and $\mathscr{S}'(\mathbb{R}^{n})$  be the dual space of $\mathscr{S}(\mathbb{R}^{n})$. The Fourier transform of a tempered distribution $f$ is denoted by $\hat{f}$,~while its inverse transform is denoted by $\check{f}$. Assume that $\varphi_{0},\varphi\in \mathscr{S}(\mathbb{R}^{n})$ with $\varphi_{0}\geq0$ satisfying the following conditions:
\begin{equation*}
\varphi_{0}(x)=
\begin{cases}
1,\ |x|\leq1,\\
0,\ |x|\geq2.
\end{cases}
\end{equation*}
For $x\in \mathbb{R}^{n}$, let $\varphi(x)=\varphi_{0}(x)-\varphi_{0}(2x)$ and  $\varphi_{j}(x)=\varphi(2^{-j}x)$ with $j\in\mathbb{Z}$. Then we call $\{\varphi_{j}\}_{j\in\mathbb{Z}}$ a resolution of unity. The following definition comes from \cite{SCJS}.

\begin{defn}
Let $\alpha, \beta\in \mathbb{R}$, $0<r, q\leq\infty$, $p(\cdot)\in\mathscr{P}(\mathbb{R}^{n})$ and $\{\varphi_{j}\}_{j\in \mathbb{Z}}$ be resolution of unity.
\begin{enumerate}
\item[\rm (i)]
The homogeneous Herz Triebel-Lizorkin spaces with variable exponent $\dot{K}^{\alpha,q}_{p(\cdot)}\dot{F}^{\beta}_{r}(\mathbb{R}^{n})$ is defined by
\begin{equation*}
\left\{f\in \mathscr{S}(\mathbb{R}^{n}): \|f\|_{\dot{K}^{\alpha,q}_{p(\cdot)}\dot{F}^{\beta}_{r}(\mathbb{R}^{n})}<\infty\right\},
\end{equation*}
where
\begin{equation*}
\|f\|_{\dot{K}^{\alpha,q}_{p(\cdot)}\dot{F}^{\beta}_{r}(\mathbb{R}^{n})}
=\left\|\left(\sum^{\infty}_{j=-\infty}\left|2^{\beta j}{\varphi_{j}}^{\vee}\ast f\right|^{r}\right)^{1/r}\right\|_{\dot{K}^{\alpha,q}_{p(\cdot)}}.
\end{equation*}

\item[\rm (ii)]
The non-homogeneous Herz Triebel-Lizorkin spaces with variable exponent $K^{\alpha,q}_{p(\cdot)}\dot{F}^{\beta}_{r}(\mathbb{R}^{n})$ is defined by
\begin{equation*}
\left\{f\in \mathscr{S}(\mathbb{R}^{n}): \|f\|_{K^{\alpha,q}_{p(\cdot)}\dot{F}^{\beta}_{r}(\mathbb{R}^{n})}<\infty\right\},
\end{equation*}
where
\begin{equation*}
\|f\|_{K^{\alpha,q}_{p(\cdot)}\dot{F}^{\beta}_{r}(\mathbb{R}^{n})}
=\left\|\left(\sum^{\infty}_{j=-\infty}\left|2^{\beta j}{\varphi_{j}}^{\vee}\ast f\right|^{r}\right)^{1/r}\right\|_{K^{\alpha,q}_{p(\cdot)}}.
\end{equation*}
\end{enumerate}
\end{defn}

Note that a operator $T$ is called a  sublinear operator if, for any integrable functions $f,g$ and constant $\lambda\in \mathbb{R}$, $T$ stisfying
\begin{align}\label{sub-def}
T(f+g)\leq Tf+Tg \quad \text{and} \quad \lambda Tf{\leq T(\lambda f)}.
\end{align}
Based on the above definition, we now give the main results of this paper as follows.
\begin{thm}\label{H1}
Let $\alpha\in \mathbb{R}$, $0<q\leq\infty$ and $p(\cdot)\in\mathscr{B}(\mathbb{R}^{n})$.
Suppose that $\delta_{1}, \delta_{2}\in(0,1)$ satisfy (\ref{1.1}) and $-n\delta_{1}<\alpha<n\delta_{2}$.
Assume further that $b\in \dot{\Lambda}_{\beta}$ with $0<\beta<1$.
If a sublinear operator $T$ satisfies the size condition
\begin{align}\label{size-1}
|Tf(x)|\leq C\int_{\mathbb{R}^{n}}\frac{|f(y)|}{|x-y|^{n}}dy,~x\not\in\supp f,
\end{align}
for any integrable function $f$ with compact support and $T$ is bounded on $L^{p(\cdot)}(\mathbb{R}^{n})$,
then
\begin{enumerate}
\item[\rm (i)]
$[b,T]$ is a bounded from $\dot{K}^{\alpha,q}_{p(\cdot)}(\mathbb{R}^{n})$ to $\dot{K}^{\alpha,q}_{p(\cdot)}\dot{F}^{\beta}_{\infty}(\mathbb{R}^{n})$;

\item[\rm (ii)]
$[b,T]$ is a bounded from $K^{\alpha,q}_{p(\cdot)}(\mathbb{R}^{n})$ to $K^{\alpha,q}_{p(\cdot)}\dot{F}^{\beta}_{\infty}(\mathbb{R}^{n})$.
\end{enumerate}
\end{thm}

\begin{thm}\label{H2}
Let $\alpha\in \mathbb{R}$, $0<\lambda<n$, $0<q_{1},q_{2}<\infty$, $p_{1}(\cdot),p_{2}(\cdot)\in\mathscr{B}(\mathbb{R}^{n})$ with $1/p_{2}(\cdot)=1/p_{1}(\cdot)-\lambda/n$.
Assume that $\delta_{1}, \delta_{2}\in(0,1)$ satisfy (\ref{1.1}) and $-n\delta_{1}<\alpha<n\delta_{2}$.
If a sublinear operator $T_{\lambda}$ satisfies the condition
\begin{align}\label{size-2}
|T_{\lambda}(f)(x)|\leq C\int_{\mathbb{R}^{n}}\frac{|f(y)|}{|x-y|^{n-\lambda}}dy,~x\not\in \supp f,
\end{align}
for any integrable function $f$ with compact support and $T_{\lambda}$ is a bounded operator from $L^{p_{1}(\cdot)}(\mathbb{R}^{n})$ to $L^{p_{2}(\cdot)}(\mathbb{R}^{n})$,
then for $b\in \dot{\Lambda}_{\beta}$ with $0<\beta<1$,
\begin{enumerate}
\item[\rm (i)]
$[b,T_{\lambda}]$ is a bounded from $\dot{K}^{\alpha, q_{1}}_{p_{1}(\cdot)}(\mathbb{R}^{n})$ to $\dot{K}^{\alpha,q_{2}}_{p_{2}(\cdot)}\dot{F}^{\beta}_{\infty}(\mathbb{R}^{n})$;

\item[\rm (ii)]
$[b,T_{\lambda}]$ is a bounded from $K^{\alpha, q_{1}}_{p_{1}(\cdot)}(\mathbb{R}^{n})$ to $K^{\alpha,q_{2}}_{p_{2}(\cdot)}\dot{F}^{\beta}_{\infty}(\mathbb{R}^{n})$.
\end{enumerate}
\end{thm}

This paper is organized as follows.
In Section 2, by the properties of Peetre maximal operator on the homogeneous spaces $\dot{K}^{\alpha,q}_{p(\cdot)}\dot{F}^{\beta}_{r}(\mathbb{R}^{n})$ and Hardy-Littlewood maximal operator on the homogeneous spaces $\dot{K}^{\alpha,q}_{p(\cdot)}(\mathbb{R}^{n})$ (see Lemma \ref{L1} and Lemma \ref{lem-M-K}),
we characterize the homogeneous spaces $\dot{K}^{\alpha,q}_{p(\cdot)}\dot{F}^{\beta}_{r}(\mathbb{R}^{n})$ via two families of operators in Theorem \ref{the-KF}.
Using a characterization of the homogeneous spaces $\dot{K}^{\alpha,q}_{p(\cdot)}\dot{F}^{\beta}_{\infty}(\mathbb{R}^{n})$, we give the proofs of Theorem \ref{H1} and Theorem \ref{H2} in Section 3 and Section 4, respectively.
Finally, as an application, some classical sublinear operators suitable for Theorem \ref{H1} and Theorem \ref{H2} are introduced in Section 5.

Throughout this paper, we will adopt the folowing notations.
Let $\mathbb{N}=\{1,2,\dots\}$ and $\mathbb{Z}=\{\dots, -1, 0, 1, \dots\}$.
For set $E$, let $|E|$ be the Lebesgue measure and $\chi_E$ be the characteristic function.
We denote by $C$  {positive constants}, which
are independent of the main parameters, but they may vary from line to
line. If $f\le Cg$, we then write $f\ls g$ and $g\gs f$;
and if $f \ls g\ls f$, we  write $f\sim g$.

\section{Characterizations of the homogeneous spaces $\dot{K}^{\alpha,q}_{p(\cdot)}\dot{F}^{\beta}_{r}(\mathbb{R}^{n})$}
\quad
It is worth that Wei and Zhang \cite[p.6]{WZ} characterized the non-homogeneous spaces $K^{\alpha,q}_{p(\cdot)}\dot{F}^{\beta}_{r}(\mathbb{R}^{n})$,
which implies the following fact:
\begin{equation}\label{eq-H3}
\|f\|_{K^{\alpha,q}_{p(\cdot)}\dot{F}^{\beta}_{\infty}(\mathbb{R}^{n})}
\sim\left\|\sup_{Q}\frac{1}{|Q|^{1+\beta/n}}\int_{Q}\mid f-f_{Q}\mid\right\|_{K^{\alpha,q}_{p(\cdot)}(\mathbb{R}^{n})},
\end{equation}
where $\alpha\in \mathbb{R}$, $0<\beta<1$, $0<q\leq\infty$ and $p(\cdot)\in\mathscr{B}(\mathbb{R}^{n})$.
Correspondingly, we want to show the characterizations of the homogeneous spaces $\dot{K}^{\alpha,q}_{p(\cdot)}\dot{F}^{\beta}_{r}(\mathbb{R}^{n})$ in this section. To this ends,  we first recall some definitions and lemmas.

Let $\varepsilon>0$, integer $S\geq-1$ and $\Psi_{0},\Psi\in\mathscr{S}(\mathbb{R}^{n})$ satisfying
\begin{equation}\label{2.1}
|\hat{\Psi}_{0}(\xi)|>0~~~on~\{|\xi|<2\varepsilon\},
\end{equation}
\begin{equation}\label{2.2}
|\hat{\Psi}(\xi)|>0~~~on~\{\frac{\varepsilon}{2}<|\xi|<2\varepsilon\},
\end{equation}
\begin{equation}\label{2.3}
D^{\tau}\hat{\Psi}(0)=0,~~\forall|\tau|\leq S.
\end{equation}
Here,\eqref{2.1} and \eqref{2.2} are Tauberian conditions, \eqref{2.3} expresses moment condition of $\Psi$.

\begin{defn}
Let $a>0$, $f\in\mathscr{S'}(\mathbb{R}^{n})$ and $\{\Psi_{j}\}_{j\in\mathbb{Z}}\subset\mathscr{S}(\mathbb{R}^{n})$. The Peetre's maximal functions are defined by
\begin{equation*}
(\Psi^{\ast}_{j})_{a}(x)=\sup_{y\in\mathbb{R}^{n}}\frac{|\Psi_{j}\ast f(y)|}{(1+|2^{j}(x-y)|)^{a}}\quad \text{for all} \ x\in \mathbb{R}^{n}.
\end{equation*}
\end{defn}

\begin{lem}[\cite{SCJS}]\rm\label{L1}
Let $a\in \mathbb{R}$, $\beta<S+1$,~$0<r,q\leq\infty$, $p(\cdot)\in\mathscr{P}(\mathbb{R}^{n})$
and $p(\cdot)/p_{0}\in \mathscr{B}(\mathbb{R}^{n})$ with $p_{0}<\min(p_{-},1)$.
Suppose that  $\delta_{1},\delta_{2}\in(0, 1)$ satisfy (\ref{1.1}) for $p(\cdot)/p_{0}$ and $-n\delta_{1}<\alpha p_{0}<n\delta_{2}$.
Assume further that $\Phi_{0}$ and $\Phi$ belong to $\mathscr{S}(\mathbb{R}^{n})$ be given by \eqref{2.1} and \eqref{2.2}, repectively.
If $a>n/p_{0}$ and $p_{0}<\beta$, then
\begin{equation*}
\|f\|_{\dot{K}^{\alpha,q}_{p(\cdot)}\dot{F}^{\beta}_{r}(\mathbb{R}^{n})}
\sim\left\|\left(\sum^{\infty}_{k=-\infty}[2^{k\beta r}(\Phi^{\ast}_{k})_{a}f]^{r}\right)^{1/r}\right\|_{\dot{K}^{\alpha,q}_{p(\cdot)}(\mathbb{R}^{n})}
\sim\left\|\left(\sum^{\infty}_{k=-\infty}2^{k\beta r}|\Phi_{k}\ast f|^{r}\right)^{1/r}\right\|_{\dot{K}^{\alpha,q}_{p(\cdot)}(\mathbb{R}^{n})}.
\end{equation*}
\end{lem}

\begin{lem}[\cite{SCJS}]\rm\label{lem-M-K}
Let $p(\cdot)\in \mathscr{B}(\mathbb{R}^{n})$, $0<q<\infty$ and $1<r<\infty$.
Assume that  $\delta_{1},\delta_{2}\in(0, 1)$ satisfy (\ref{1.1}) and $-n\delta_{1}<\alpha<n\delta_{2}$.
Then there exists a constant C such that for all sequences $\{f_{j}\}^{\infty}_{j=-\infty}$ of locally  functions on $\mathbb{R}^{n}$,
\begin{equation*}
\left\|\left(\sum^{\infty}_{k=-\infty}|Mf_{k}|^{r}\right)^{1/r}\right\|_{\dot{K}^{\alpha,q}_{p(\cdot)}(\mathbb{R}^{n})}\leq C\left\|\left(\sum^{\infty}_{k=-\infty}|f_{k}|^{r}\right)^{1/r}\right\|_{\dot{K}^{\alpha,q}_{p(\cdot)}(\mathbb{R}^{n})}.
\end{equation*}
\end{lem}

The following two families of operators are introduced by \cite{S}.

\begin{defn}
Let $Q_{z}(t)=Q(z,t)$ be a cube centered at $z\in\rn$, with side length $t$ and sides parallel to the axes.
For $m\in \mathbb{N}$, $\beta\in \mathbb{R}$ and $0<\gamma\leq\infty$.

(i)The family of operators $S^{\beta}_{r,\gamma,m}$ are defined by
\begin{equation*}
S^{\beta}_{r,\gamma,m}f(x)=\left(\int^{\infty}_{0}\left(\frac{1}{|Q_{0}(t)|}\int_{Q_{0}(t)}|\Delta^{m}_{h}f(x)|^{\gamma}dh\right)^{r/\gamma}\frac{dt}{t^{1+\beta r}}\right)^{1/r}\quad \text{for all} \ x\in\rn,
\end{equation*}
where $\Delta^{m}_{h}$~is the difference operator, that is,
\begin{equation*}
\Delta^{1}_{h}f(x)=\Delta_{h}f(x)=f(x+h)-f(x),
\end{equation*}
\begin{equation*}
\Delta^{m+1}_{h}f(x)=\Delta^{m}_{h}f(x+h)-\Delta^{m}_{h}f(x), \quad m\geq1.
\end{equation*}

(ii)For a fixed ball $Q=Q_{x}(t)$ with $x\in\rn$ and $t>0$, we define the oscillation
\begin{equation*}
osc^{m}_{\gamma}(f,Q)=osc^{m}_{\gamma}(f,x,t)=\inf_{P\in P^{m}}\left(\frac{1}{|Q|}\int_{Q}|f(y)-P(y)|^{\gamma}dy\right)^{1/\gamma},
\end{equation*}
where the infimum is taken over all polynomials of degree not exceeding $m$. Further, we define the family of operators
\begin{equation*}
\mathscr{D}^{\beta}_{r,\gamma,m}f(x)=\left(\int^{\infty}_{0}(osc^{m-1}_{\gamma}(f,x,t))^{r}\frac{dt}{t^{1+\beta r}}\right)^{1/r}.
\end{equation*}
For $q=\infty$ or $r=\infty$, we have the usual modifications and replace integrations by sup-norms.
\end{defn}

Using Lemma \ref{L1} and Lemma \ref{lem-M-K}, we repeat the process of \cite[pp.5-6]{WZ} and obtain the following theorem.
The details being omitted.
\begin{thm}\label{the-KF}
Let $a\in \mathbb{R}$, $0<r<\infty$ and $p(\cdot)\in\mathscr{B}(\mathbb{R}^{n})$.
Assume that $\delta_{1}, \delta_{2}\in(0,1)$ satisfy (\ref{1.1}) and $-n\delta_{1}<\alpha<n\delta_{2}$.
If $1<\gamma\leq\infty$, $m>\frac{\beta}{a_{0}}$,
\begin{equation*}
\beta>\sigma_{p,r,\gamma}=\max\left\{0,\nu\left(\frac{1}{p_{-}}-\frac{1}{\gamma}\right),\nu\left(\frac{1}{r}-\frac{1}{\gamma}\right)\right\}
\end{equation*}
and $\nu$ is the trace of a matrix $(see \cite[pp.\,391]{S})$, then
\begin{equation*}
\|f\|_{\dot{K}^{\alpha,q}_{p(\cdot)}\dot{F}^{\beta}_{r}(\mathbb{R}^{n})}\sim\|S^{\beta}_{r,\gamma,m}f\|_{\dot{K}^{\alpha,q}_{p(\cdot)}(\mathbb{R}^{n})}
\sim\|\mathscr{D}^{\beta}_{r,\gamma,m}f\|_{\dot{K}^{\alpha,q}_{p(\cdot)}(\mathbb{R}^{n})}.
\end{equation*}
\end{thm}

\begin{rem}
Obviously, if we take $r=\infty$ in Theorem \ref{the-KF}, then
\begin{equation}\label{eq-H3-dot}
\|f\|_{\dot{K}^{\alpha,q}_{p(\cdot)}\dot{F}^{\beta}_{\infty}(\mathbb{R}^{n})}
\sim\left\|\sup_{Q}\frac{1}{|Q|^{1+\beta/n}}\int_{Q}\mid f-f_{Q}\mid\right\|_{\dot{K}^{\alpha,q}_{p(\cdot)}(\mathbb{R}^{n})}.
\end{equation}
\end{rem}

\section{The proof of Theorem \ref{H1}}
\quad
This section is devote to the proof of Theorem \ref{H1}. Now, we recall two lemmas.
Note that the following lemma may be found in \cite[p14, p38]{RA} and \cite{JTW}.

\begin{lem}\label{L3}
Let $0<\beta<1$ and $1<q\leq \infty$. Then
\begin{equation*}
\|f\|_{\dot{\Lambda}_{\beta}}\sim\sup_{Q}\frac{1}{|Q|^{1+\beta/n}}\int_{Q}|f-f_{Q}|\sim\sup_{Q}\frac{1}{|Q|^{\beta/n}}(\frac{1}{|Q|}\int_{Q}|f-f_{Q}|^{q})^{1/q},
\end{equation*}
for $q=\infty$~the formula should be interpreted appropriately.
\end{lem}

\begin{lem}[\cite{IM1}]\rm\label{L4}
Let $a\in \mathbb{R}$, $0<q<\infty$ and $p(\cdot)\in \mathscr{B}(\mathbb{R}^{n})$. Suppose that $\delta_{1}, \delta_{2}\in (0, 1)$ satisfy (\ref{1.1}) and $-n\delta_{1}<\alpha<n\delta_{2}$. Assume that $T$ is a sublinear operator satisfying the size condition \eqref{size-1}.
If $T$ is bounded operator on $L^{p(\cdot)}(\mathbb{R}^{n})$, then $T$ is bounded on $\dot{K}^{\alpha,q}_{p(\cdot)}(\mathbb{R}^{n})$ and $K^{\alpha,q}_{p(\cdot)}(\mathbb{R}^{n})$.
\end{lem}

\begin{proof}[\bf{Proof of Theorem \ref{H1}}]
Fix a cube $Q=Q(c_{Q},t)$ with $c_{Q}\in\rn$ and $t>0$. For $f\in L^{p(\cdot)}(\mathbb{R}^{n})$, define $f_{1}=f\chi_{Q(c_{Q},3t)}$ and $f_{2}=f-f_{1}$. Let
$3^{k}Q=Q(c_{Q}, 3^{k}t)$ and $z\in3^{k}Q\backslash(3^{k-1}Q)$ for $k\geq2$.
Then there exists $x_{Q}\in Q$ such that
\begin{align}\label{eq-xQ}
|x_{Q}-z|>t/2+3^{k-1}t>3\cdot3^{k-2}t=\inf_{y\in3Q}|y-z|\quad \text{for} \ k\geq2.
\end{align}

Define
$$C_{b, f, Q}(x_{Q}):=3\int_{\rn\backslash(3Q)}\frac{|b(z)-b_{3Q}||f(z)|}{|x_{Q}-z|^{n}}dz.$$
Then
\begin{align}\label{eq-f123}
&\frac{1}{|Q|^{1+\beta/n}}\int_{Q}|[b,T]f(y)-([b,T]f)_{Q}|dy\notag\\
&\leq\frac{2}{|Q|^{1+\beta/n}}\int_{Q}|[b,T]f(y)-C_{b, f, Q}(x_{Q})|dy\notag\\
&=\frac{2}{|Q|^{1+\beta/n}}\int_{Q}|b(y)Tf(y)-b_{3Q}Tf(y)
+b_{3Q}Tf(y)-T(bf)(y)-C_{b, f, Q}(x_{Q})|dy\notag\\
&\lesssim\frac{1}{|Q|^{1+\beta/n}}\left(\int_{3Q}|(b(y)-b_{3Q})Tf(y)|dy
+\int_{3Q}|b_{3Q}Tf(y)-T(bf)(y)-C_{b, f, Q}(x_{Q})|dy\right)\notag\\
&=:\frac{1}{|Q|^{1+\beta/n}}(\rm{F_{1}+F_{2}}).
\end{align}

Firstly, we consider $\rm{F_{1}}$. From Lemma \ref{L3}, it follows that
\begin{equation*}
\begin{aligned}
{\rm{F_{1}}}
\lesssim\sup_{y\in 3Q}|b(y)-b_{3Q}|\int_{3Q}|Tf(y)|dy
\lesssim\|b\|_{\dot{\Lambda}_{\beta}}|3Q|^{1+\beta/n}M(Tf)(x)
\quad \text{for all} \ x\in3Q.
\end{aligned}
\end{equation*}

To estimate $\rm{F_{2}}$,  we consider the following two cases: $b_{3Q}Tf(y)-T(bf)(y)-C_{b, f, Q}\geq0$ and $b_{3Q}Tf(y)-T(bf)(y)-C_{b, f, Q}<0$ for $y\in3Q$, where
$$C_{b, f, Q}:=\sup_{y\in3Q}\int_{\rn\backslash(3Q)}\frac{|b(z)-b_{3Q}||f(z)|}{|y-z|^{n}}dz.$$

{\it Case 1:\, Estimation of $\rm{F_{2}}$ in the case $b_{3Q}Tf(y)-T(bf)(y)-C_{b, f, Q}\geq0$ with $y\in3Q$.}  By the definition of sublinear operator in \eqref{sub-def},
we deduce
\begin{align*}
b_{3Q}Tf(y)
&\leq T(b_{3Q}f)(y)\\
&\leq T((b_{3Q}-b)(f_{1}+f_{2}))(y)+T(bf)(y)\\
&\leq T((b_{3Q}-b)f_{1})(y)+T((b_{3Q}-b)f_{2})(y)+T(bf)(y)\\
&\leq T((b_{3Q}-b)f_{1})(y)+C_{b, f, Q}+T(bf)(y),
\end{align*}
where the last estimate used the condition \eqref{size-1}. So,
\begin{align*}
0<b_{3Q}Tf(y)-T(bf)(y)-C_{b, f, Q}\leq T((b_{3Q}-b)f_{1})(y).
\end{align*}
This and the fact $C_{b, f, Q}>C_{b, f, Q}(x_{Q})$ by terms of \eqref{eq-xQ} allow that
\begin{align}\label{1-F}
{\rm{F_{2}}}
&=\int_{3Q}|b_{3Q}Tf(y)-T(bf)(y)-C_{b, f, Q}+C_{b, f, Q}-C_{b, f, Q}(x_{Q})|dy\notag\\
&\leq\int_{3Q}|b_{3Q}Tf(y)-T(bf)(y)-C_{b, f, Q}|dy+\int_{3Q}|C_{b, f, Q}-C_{b, f, Q}(x_{Q})|dy\notag\\
&\leq\int_{3Q}|T((b-b_{3Q})f_{1})(y)|dy+|3Q|(C_{b, f, Q}-C_{b, f, Q}(x_{Q}))\\
&=:{\rm{F_{21}}+\rm{F_{22}}}\notag.
\end{align}

To estimate $\rm{F_{21}}$, using H\"{o}lder's inequality and the boundedness of $T$ on $L^{p(\cdot)}({\mathbb{R}^n})$,
we see that for any $x\in3Q$ and for some $t$ satisfying $1<t<p^{-},$
\begin{equation*}
\begin{aligned}
{\rm{F_{21}}}
&\leq|3Q|^{1-1/t}\left(\int_{3Q}|T((b-b_{Q})f_{1})(y)|^{t}dy\right)^{1/t}\\
&\lesssim |3Q|^{1-1/t}\left(\int_{3Q}|(b-b_{Q})f(y)|^{t}dy\right)^{1/t}\\
&\lesssim|3Q|\sup_{y\in Q}|b(y)-b_{Q}|\left(\frac{1}{|Q|}\int_{Q}|f(y)|^{t}dy\right)^{1/t}\\
&\lesssim|3Q|^{1+\beta/n}\|b\|_{\dot{\Lambda}_{\beta}}(M(|f|^{t}))^{1/t}(x).
\end{aligned}
\end{equation*}

To estimate $\rm{F_{22}}$, we need the following well known fact. Let $Q^{\ast}\subset Q$. Then
\begin{equation*}
|b_{Q^{\ast}}-b_{Q}|\lesssim\|b\|_{\dot{\Lambda}_{\beta}}|Q|^{\beta/n}.
\end{equation*}
This and $C_{b, f, Q}>C_{b, f, Q}(x_{Q})>C_{b, f, Q}(x_{Q})/3>0$ imply that
\begin{equation*}
\begin{aligned}
{\rm{F_{22}}}
&\leq |3Q|(C_{b, f, Q}-C_{b, f, Q}(x_{Q})/3)\\
&=|3Q|\sum^{\infty}_{k=2}\int_{3^{k}Q\backslash3^{k-1}Q}\sup_{y\in3Q}\left(\frac{1}{(|y-z|)^{n}}-\frac{1}{(|x_{Q}-z|)^{n}}\right)|b(z)-b_{3Q}||f(z)|dz\\
&\lesssim|3Q|\sum^{\infty}_{k=2}\int_{3^{k}Q\backslash3^{k-1}Q}\sup_{y\in3Q}\frac{|y-x_{Q}|^{n}}{|x_{Q}-z|^{2n}}
(|b(z)-b_{3^{k}Q}|+|b_{3^{k}Q}-b_{3Q}|)|f(z)|dz\\
&\lesssim|3Q|\sum^{\infty}_{k=2}\frac{|3Q|}{|3^{k}Q|^{2}}|3^{k}Q|^{\beta/n}\|b\|_{\dot{\Lambda}_{\beta}}\int_{3^{k}Q}|f(z)|dz\\
&\lesssim|3Q|^{1+\beta/n}\|b\|_{\dot{\Lambda}_{\beta}}\sum^{\infty}_{k=2}3^{(k-1)(\beta-1)}M(f)(x)\\
&\lesssim|3Q|^{1+\beta/n}\|b\|_{\dot{\Lambda}_{\beta}}M(f)(x).
\end{aligned}
\end{equation*}
for all $x\in3^{k}Q$ with $k\geq2$. Therefore, combing the estimates of $\rm{F_{21}}$  and $\rm{F_{22}}$, we obtain that $${\rm{F_{2}}}\lesssim|3Q|^{1+\beta/n}\|b\|_{\dot{\Lambda}_{\beta}}M(f)(x) \quad \text{for all} \ x\in3Q.$$

{\it Case 2:\, Estimation of $\rm{F_{2}}$ in the case $b_{3Q}Tf(y)-T(bf)(y)-C_{b, f, Q}<0$ with $y\in3Q$.}
According to the definition of sublinear operator  in \eqref{sub-def} and the condition \eqref{size-1}, we have
\begin{align*}
&|b_{3Q}Tf(y)-T(bf)(y)-C_{b, f, Q}|\\
&=C_{b, f, Q}+T(bf)(y)-b_{3Q}Tf(y)\\
&\leq C_{b, f, Q}+\left(T(bf)(y)-Tf(y)\sup_{Q}b_{3Q}\right)+\left(Tf(y)\sup_{Q}b_{3Q}-b_{3Q}Tf(y)\right)\\
&\leq C_{b, f, Q}+\left(T\left(\sup_{Q}|b-b_{3Q}|f\right)(y)\right)+|Tf(y)|\sup_{Q}\frac{1}{|3Q|}\int_{3Q}b(z)-b_{3Q}dz\\
&\lesssim C_{b, f, Q}+\sup_{Q}\sup_{y\in3Q}\int_{\rn\backslash(3Q)}\frac{|b(z)-b_{3Q}||f(z)|}{|y-z|}dz
+|3Q|^{\beta/n}\|b\|_{\dot{\Lambda}_{\beta}}|Tf(y)|\\
&\lesssim 2\sup_{Q}C_{b, f, Q}+|3Q|^{\beta/n}\|b\|_{\dot{\Lambda}_{\beta}}|Tf(y)|.
\end{align*}
This implies that
\begin{align}\label{2-F}
{\rm{F_{2}}}
&=\int_{3Q}|b_{3Q}Tf(y)-T(bf)(y)-C_{b, f, Q}+C_{b, f, Q}-C_{b, f, Q}(x_{Q})|dy\notag\\
&\leq\int_{3Q}|b_{3Q}Tf(y)-T(bf)(y)-C_{b, f, Q}|dy+\int_{3Q}|C_{b, f, Q}-C_{b, f, Q}(x_{Q})|dy\notag\\
&\lesssim |3Q|2\sup_{Q}C_{b, f, Q}+|3Q||3Q|^{\beta/n}\|b\|_{\dot{\Lambda}_{\beta}}|Tf(y)|+|3Q|(C_{b, f, Q}-C_{b, f, Q}(x_{Q}))\notag\\
&\lesssim |3Q|(3\sup_{Q}C_{b, f, Q}-C_{b, f, Q}(x_{Q}))+|3Q|^{1+\beta/n}\|b\|_{\dot{\Lambda}_{\beta}}|Tf(y)|.
\end{align}
where in the second estimate we used the fact $C_{b, f, Q}>C_{b, f, Q}(x_{Q})$ by terms of \eqref{eq-xQ}.
Note that we repeat the process of estimation of ${\rm{F_{22}}}$ and obtain
$$|3Q|(3\sup_{Q}C_{b, f, Q}-C_{b, f, Q}(x_{Q}))\ls|3Q|^{1+\beta/n}\|b\|_{\dot{\Lambda}_{\beta}}M(f)(x).$$
Therefore,
\begin{align*}
{\rm{F_{2}}}\lesssim|3Q|^{1+\beta/n}\|b\|_{\dot{\Lambda}_{\beta}}(M(f)(x)+|Tf(y)|) \quad \text{for} \ x, y\in3Q.
\end{align*}
This ends the estimate of {\it Case 2}.

Summarizing all the  estimates of {\it Case 1}  and {\it Case 2}, we conclude that
\begin{align*}
{\rm{F_{2}}}\lesssim|3Q|^{1+\beta/n}\|b\|_{\dot{\Lambda}_{\beta}}\left(M(f)(x)+(M(|f|^{t}))^{1/t}(x)+|Tf(y)|\right) \quad \text{for all} \ x, y\in3Q.
\end{align*}
This and the  estimate of ${\rm F}_{1}$ allow that
\begin{equation*}
\begin{aligned}
\frac{1}{|Q|^{1+\beta/n}}\int_{Q}|[b,T]f-([b,T]f)_{Q}|
\lesssim\|b\|_{\dot{\Lambda}_{\beta}}\left(M(f)(x)+(M(|f|^{t}))^{1/t}(x)+|Tf(y)|\right) \quad \text{for all} \ x, y\in3Q.
\end{aligned}
\end{equation*}
Taking the supremum over all $Q$ such that $x\in Q$ on both sides, it follows from \eqref{eq-H3} and \eqref{eq-H3-dot} that
\begin{equation*}
\begin{aligned}
\|[b,T]f\|_{K^{\alpha,q}_{p(\cdot)}\dot{F}^{\beta}_{\infty}(\mathbb{R}^{n})}
\lesssim\|b\|_{\dot{\Lambda}_{\beta}}\left(\|Mf\|_{K^{\alpha,q}_{p(\cdot)}(\mathbb{R}^{n})}+\|(M(|f|^{t}))^{1/t}\|_{K^{\alpha,q}_{p(\cdot)}(\mathbb{R}^{n})}
+\|Tf\|_{K^{\alpha,q}_{p(\cdot)}(\mathbb{R}^{n})}\right)
\end{aligned}
\end{equation*}
and
\begin{equation*}
\begin{aligned}
\|[b,T]f\|_{\dot{K}^{\alpha,q}_{p(\cdot)}\dot{F}^{\beta}_{\infty}(\mathbb{R}^{n})}
\lesssim\|b\|_{\dot{\Lambda}_{\beta}}\left(\|Mf\|_{\dot{K}^{\alpha,q}_{p(\cdot)}(\mathbb{R}^{n})}
+\|(M(|f|^{t}))^{1/t}\|_{\dot{K}^{\alpha,q}_{p(\cdot)}(\mathbb{R}^{n})}
+\|Tf\|_{\dot{K}^{\alpha,q}_{p(\cdot)}(\mathbb{R}^{n})}\right).
\end{aligned}
\end{equation*}
Further, by Lemma 3.2 and the facts that $M$ is bounded on $K^{\alpha,q}_{p(\cdot)}(\mathbb{R}^{n})$ and $\dot{K}^{\alpha,q}_{p(\cdot)}(\mathbb{R}^{n})$, we obtain that (i) and (ii) of Theorem \ref{H1} hold.
\end{proof}

\section{The proof of Theorem \ref{H2}}
\quad
The section is devote to the proof of Theorem \ref{H2}.
To this ends, we show the boundedness of fractional type sublinear operators on Herz spaces with variable exponent in Theorem \ref{T-HH}.

\begin{lem}[\cite{IM1}]\rm\label{L2}
Let $p(\cdot)\in\mathscr{B}(\mathbb{R}^{n})$. Then there exist a constant $C>0$ such that for all ball $B$ in $\mathbb{R}^{n}$
\begin{equation}\label{1.2}
\frac{1}{|B|}\|\chi_{B}\|_{L^{p(\cdot)}(\mathbb{R}^{n})}\|\chi_{B}\|_{L^{p'(\cdot)}(\mathbb{R}^{n})}\leq C.
\end{equation}
\end{lem}

\begin{thm}\label{T-HH}
Let $a\in \mathbb{R}$, $0<\lambda<n$, $0<q_{1}<q_{2}\leq\infty,$ $p_{1}(\cdot),p_{2}(\cdot)\in\mathscr{B}(\mathbb{R}^{n})$ and $1/p_{2}(\cdot)=1/p_{1}(\cdot)-\lambda/n$. Assume that $\delta_{1} ,\delta_{2}\in(0, 1)$ satisfy \eqref{1.1} and
$-n\delta_{1}<\alpha<n\delta_{2}$. If a sublinear operator $T_{\lambda}$ satisfies \eqref{size-2} and $T_{\lambda}$ is bounded from $L^{p_{1}(\cdot)}(\mathbb{R}^{n})$ to $L^{p_{2}(\cdot)}(\mathbb{R}^{n})$, then \begin{enumerate}
\item[\rm (i)]
$T_{\lambda}$ is bounded from $\dot{K}^{\alpha,q_{1}}_{p_{1}(\cdot)}(\mathbb{R}^{n})$ to $\dot{K}^{\alpha,q_{2}}_{p_{2}(\cdot)}(\mathbb{R}^{n})$;

\item[\rm (ii)]
$T_{\lambda}$ is bounded from $K^{\alpha,q_{1}}_{p_{1}(\cdot)}(\mathbb{R}^{n})$ to $K^{\alpha,q_{2}}_{p_{2}(\cdot)}(\mathbb{R}^{n})$.
\end{enumerate}
\end{thm}
\begin{proof}
To show (i), for any $0<q_{1}<q_{2}<1$ and $0<q_{1}/q_{2}<1$, by
\begin{equation*}
\left(\sum^{\infty}_{i=1}a_{i}\right)^{q_{1}/q_{2}}\leq\sum^{\infty}_{i=1}(a_{i})^{q_{1}/q_{2}}, (a_{1},a_{2},a_{3}...>0),
\end{equation*}
we deduce
\begin{equation*}
\begin{aligned}
\|T_{\lambda}f\|^{q_{1}}_{\dot{K}^{\alpha,q_{2}}_{p_{2}(\cdot)}(\mathbb{R}^{n})}
&=\left\{\sum^{\infty}_{k=-\infty}2^{k\alpha q_{2}}\|(T_{\lambda}f)\chi_{k}\|^{q_{2}}_{L^{p_{2}(\cdot)}(\mathbb{R}^{n})}\right\}^{q_{1}/q_{2}}\\
&\leq\sum^{\infty}_{k=-\infty}2^{k\alpha q_{1}}\|(T_{\lambda}f)\chi_{k}\|^{q_{1}}_{L^{p_{2}(\cdot)}(\mathbb{R}^{n})}\\
&\lesssim\sum^{\infty}_{k=-\infty}2^{k\alpha q_{1}}\left\|\left(\sum^{k-2}_{j=-\infty}|T_{\lambda}(f\chi_{j})|\right)\chi_{k}\right\|^{q_{1}}_{L^{p_{2}(\cdot)}(\mathbb{R}^{n})}\\
&\qquad+\sum^{\infty}_{k=-\infty}2^{k\alpha q_{1}}\left\|\left(\sum^{k+1}_{j=k-1}|T_{\lambda}(f\chi_{j})|\right)\chi_{k}\right\|^{q_{1}}_{L^{p_{2}(\cdot)}(\mathbb{R}^{n})}\\
&\qquad+\sum^{\infty}_{k=-\infty}2^{k\alpha q_{1}}\left\|\left(\sum^{\infty}_{j=k+2}|T_{\lambda}(f\chi_{j})|\right)\chi_{k}]\right\|^{q_{1}}_{L^{p_{2}(\cdot)}(\mathbb{R}^{n})}\\
&=:\rm{I_{1}+I_{2}+I_{3}}.
\end{aligned}
\end{equation*}

First, we estimate $\rm{I_{2}}$. From the fact that $T_{\lambda}f$ is bounded from $L^{p_{1}(\cdot)}(\mathbb{R}^{n})$ to $L^{p_{2}(\cdot)}(\mathbb{R}^{n})$,
it follows that
\begin{align*}
\rm{I_{2}}&=\sum^{\infty}_{k=-\infty}2^{k\alpha q_{1}}\left\|\left(\sum^{k+1}_{j=k-1}|T_{\lambda}(f\chi_{j})|\right)\chi_{k}\right\|^{q_{1}}_{L^{p_{2}(\cdot)}(\mathbb{R}^{n})}\\
&\lesssim\sum^{\infty}_{k=-\infty}2^{k\alpha q_{1}}\left\|\left(\sum^{k+1}_{j=k-1}(f\chi_{j})\right)\chi_{k}\right\|^{q_{1}}_{L^{p_{1}(\cdot)}(\mathbb{R}^{n})}\\
&\lesssim\sum^{\infty}_{k=-\infty}2^{k\alpha q_{1}}\|f\chi_{k}\|^{q_{1}}_{L^{p_{1}(\cdot)}(\mathbb{R}^{n})}\\
&\lesssim\|f\|^{q_{1}}_{\dot{K}^{\alpha,q_{1}}_{p_{1}(\cdot)}(\mathbb{R}^{n})}.
\end{align*}

Now, we estimates $\rm{I_{1}}$. For each $k\in\mathbb{Z}$, $j\leq k-2$ and a.e. $x\in E_{k}=B_{k}\backslash B_{k-1}$,
by size condition of $T_{\lambda}$ and generalized H\"{o}lder's inequality, we see that
\begin{equation*}
\begin{aligned}
|T_{\lambda}(f\chi_{j})(x)|
\lesssim\int_{E_{k}}|x-y|^{-(n-\lambda)}|f(y)|dy
\lesssim2^{-k(n-\lambda)}\|f\chi_{j}\|_{L^{p_{1}(\cdot)}(\mathbb{R}^{n})}\|\chi_{j}\|_{L^{p'_{1}(\cdot)}(\mathbb{R}^{n})}.
\end{aligned}
\end{equation*}
This implies
\begin{equation*}
\begin{aligned}
\rm{I_{1}}&= \sum^{\infty}_{k=-\infty}2^{k\alpha q_{1}}\left\|\left(\sum^{k-2}_{j=-\infty}|T_{\lambda}(f\chi_{j})|\right)\chi_{k}\right\|^{q_{1}}_{L^{p_{2}(\cdot)}(\mathbb{R}^{n})}\\
&\lesssim\sum^{\infty}_{k=-\infty}2^{k\alpha q_{1}}\left(\sum^{k-2}_{j=-\infty}2^{-k(n-\lambda)}\|f\chi_{j}\|_{L^{p_{1}(\cdot)}(\mathbb{R}^{n})}
\|\chi_{B_{j}}\|_{L^{p'_{1}(\cdot)}(\mathbb{R}^{n})}\|\chi_{B_{k}}\|_{L^{p_{2}(\cdot)}(\mathbb{R}^{n})}\right)^{q_{1}}\\
&\lesssim\sum^{\infty}_{k=-\infty}2^{k\alpha q_{1}}\left(\sum^{k-2}_{j=-\infty}\|f\chi_{j}\|_{L^{p_{1}(\cdot)}(\mathbb{R}^{n})}
\|\chi_{B_{j}}\|_{L^{p'_{1}(\cdot)}(\mathbb{R}^{n})}\|\chi_{B_{k}}\|^{-1}_{L^{p'_{1}(\cdot)}(\mathbb{R}^{n})}\right)^{q_{1}}\\
&\lesssim\sum^{\infty}_{k=-\infty}2^{k\alpha q_{1}}\left(\sum^{k-2}_{j=-\infty}\|f\chi_{j}\|_{L^{p_{1}(\cdot)}(\mathbb{R}^{n})}
\frac{\|\chi_{B_{j}}\|_{L^{p'_{1}(\cdot)}(\mathbb{R}^{n})}}{\|\chi_{B_{k}}\|_{L^{p'_{1}(\cdot)}(\mathbb{R}^{n})}}\right)^{q_{1}}\\
&\lesssim\sum^{\infty}_{k=-\infty}\left(\sum^{k-2}_{j=-\infty}2^{j\alpha}\|f\chi_{j}\|_{L^{p_{1}(\cdot)}(\mathbb{R}^{n})}
2^{(j-k)(n\delta_{2}-\alpha)}\right)^{q_{1}},
\end{aligned}
\end{equation*}
where in the antepenultimate step we used the fact
\begin{equation*}
\|\chi_{B_{k}}\|_{L^{p_{2}(\cdot)}(\mathbb{R}^{n})}
=2^{k(n-\lambda)}\|\chi_{B_{k}}\|^{-1}_{L^{p'_{1}(\cdot)}(\mathbb{R}^{n})}
\quad \text{for} \ 1/p_{2}(\cdot)=1/p_{1}(\cdot)-\lambda/n.
\end{equation*}
To continue calculations,~we consider the two cases $``0<q_{1}\leq1"$ and $``1<q_{1}<\infty"$.\\

If $0<q_{1}\leq1,n\delta_{2}-\alpha>0$,~then we get
\begin{equation*}
\begin{aligned}
\rm{I_{1}}&\lesssim\sum^{\infty}_{k=-\infty}\sum^{k-2}_{j=-\infty}2^{j\alpha q_{1}}\|f\chi_{j}\|^{q_{1}}_{L^{p_{1}(\cdot)}(\mathbb{R}^{n})}2^{(j-k)(n\delta_{2}-\alpha)q_{1}}\\
&\lesssim\sum^{\infty}_{j=-\infty}2^{j\alpha q_{1}}\|f\chi_{j}\|^{q_{1}}_{L^{p_{1}(\cdot)}(\mathbb{R}^{n})}\sum^{\infty}_{k=j+2}2^{(j-k)(n\delta_{2}-\alpha)q_{1}}\\
&\lesssim\sum^{\infty}_{j=-\infty}2^{j\alpha q_{1}}\|f\chi_{j}\|^{q_{1}}_{L^{p_{1}(\cdot)}(\mathbb{R}^{n})}\\
&\lesssim\|f\|^{q_{1}}_{\dot{K}^{\alpha,q_{1}}_{p_{1}(\cdot)}(\mathbb{R}^{n})}.\\
\end{aligned}
\end{equation*}

If $1<q_{1}<\infty,n\delta_{2}-\alpha>0$,~then we use H\"{o}lder's inequality and obtain
\begin{align*}
\rm{I_{1}}&\lesssim\sum^{\infty}_{k=-\infty}\left(\sum^{k-2}_{j=-\infty}2^{j\alpha q_{1}}\|f\chi_{j}\|^{q_{1}}_{L^{p_{1}(\cdot)}(\mathbb{R}^{n})}2^{(j-k)(n\delta_{2}-\alpha)q_{1}/2}\right)^{\frac{q_{1}}{q'_{1}}}
\left(\sum^{k-2}_{j=-\infty}2^{(j-k)(n\delta_{2}-\alpha)q'_{1}/2}\right)^{\frac{q_{1}}{q'_{1}}}\\
&\lesssim\sum^{\infty}_{k=-\infty}\sum^{k-2}_{j=-\infty}2^{j\alpha q_{1}}\|f\chi_{j}\|^{q_{1}}_{L^{p_{1}(\cdot)}(\mathbb{R}^{n})}2^{(j-k)(n\delta_{2}-\alpha)q_{1}/2}\\
&\lesssim\sum^{\infty}_{j=-\infty}2^{j\alpha q_{1}}\|f\chi_{j}\|^{q_{1}}_{L^{p_{1}(\cdot)}(\mathbb{R}^{n})}\sum^{\infty}_{k=j+2}2^{(j-k)(n\delta_{2}-\alpha)q_{1}/2}\\
&\lesssim\|f\|^{q_{1}}_{\dot{K}^{\alpha,q_{1}}_{p_{1}(\cdot)}(\mathbb{R}^{n})}.\\
\end{align*}

Next,~we estimate $\rm{I_{3}}$.
For each $k\in\mathbb{Z}$, $j\geq k+2$ and a.e. $x\in E_{k}=B_{k}\backslash B_{k-1}$, size condition of $T_{\lambda}$ and generalized H\"{o}lder's inequality imply
\begin{equation*}
\begin{aligned}
|T_{\lambda}(f\chi_{j})(x)|
\lesssim\int_{E_{k}}|x-y|^{-(n-\lambda)}|f(y)|dy
\lesssim2^{-j(n-\lambda)}\|f\chi_{j}\|_{L^{p_{1}(\cdot)}(\mathbb{R}^{n})}\|\chi_{j}\|_{L^{p'_{1}(\cdot)}(\mathbb{R}^{n})}.
\end{aligned}
\end{equation*}
Therefore
\begin{equation*}
\begin{aligned}
\rm{I_{3}}
&\lesssim\sum^{\infty}_{k=-\infty}2^{k\alpha q_{1}}\left(\sum^{\infty}_{j=k+2}2^{-j(n-\lambda)}\|f\chi_{j}\|_{L^{p_{1}(\cdot)}(\mathbb{R}^{n})}
\|\chi_{B_{j}}\|_{L^{p'_{1}(\cdot)}(\mathbb{R}^{n})}\|\chi_{B_{k}}\|_{L^{p_{2}(\cdot)}(\mathbb{R}^{n})}\right)^{q_{1}}\\
&\lesssim\sum^{\infty}_{k=-\infty}2^{k\alpha q_{1}}\left(\sum^{\infty}_{j=k+2}\|f\chi_{j}\|_{L^{p_{1}(\cdot)}(\mathbb{R}^{n})}
\frac{\|\chi_{B_{k}}\|_{L^{p_{2}(\cdot)}(\mathbb{R}^{n})}}{\|\chi_{B_{j}}\|_{L^{p_{2}(\cdot)}(\mathbb{R}^{n})}}\right)^{q_{1}}\\
&\lesssim\sum^{\infty}_{k=-\infty}\left(\sum^{\infty}_{j=k+2}2^{j\alpha}
\|f\chi_{j}\|_{L^{p_{1}(\cdot)}(\mathbb{R}^{n})}2^{(k-j)(n\delta_{1}+\alpha)}\right)^{q_{1}},
\end{aligned}
\end{equation*}
where in the penultimate step we used the fact
\begin{equation*}
\|\chi_{B_{j}}\|_{L^{p'_{1}(\cdot)}(\mathbb{R}^{n})}
=2^{j(n-\lambda)}\|\chi_{B_{j}}\|^{-1}_{L^{p_{2}(\cdot)}(\mathbb{R}^{n})}
\quad \text{for} \ 1/p_{2}(\cdot)=1/p_{1}(\cdot)-\lambda/n.
\end{equation*}
Similarly, to continue calculations, we consider the two cases $``0<q_{1}\leq1"$ and $``1<q_{1}<\infty"$.

If $0<q_{1}\leq1,n\delta_{1}+\alpha>0$, then we get
\begin{equation*}
\begin{aligned}
\rm{I_{3}}
&\lesssim\sum^{\infty}_{k=-\infty}\sum^{\infty}_{j=k+2}2^{j\alpha q_{1}}\|f\chi_{j}\|^{q_{1}}_{L^{p_{1}(\cdot)}(\mathbb{R}^{n})}2^{(k-j)(n\delta_{1}+\alpha)q_{1}}\\
&\lesssim\sum^{\infty}_{j=-\infty}2^{j\alpha q_{1}}\|f\chi_{j}\|^{q_{1}}_{L^{p_{1}(\cdot)}(\mathbb{R}^{n})}\sum^{j-2}_{k=-\infty}2^{(k-j)(n\delta_{1}+\alpha)q_{1}}\\
&\lesssim\sum^{\infty}_{j=-\infty}2^{j\alpha q_{1}}\|f\chi_{j}\|^{q_{1}}_{L^{p_{1}(\cdot)}(\mathbb{R}^{n})}\\
&\lesssim\|f\|^{q_{1}}_{\dot{K}^{\alpha,q_{1}}_{p_{1}(\cdot)}(\mathbb{R}^{n})}.\\
\end{aligned}
\end{equation*}

If $1<q_{1}<\infty,n\delta_{1}+\alpha>0$,~then we use H\"{o}lder's inequality and obtain
\begin{equation*}
\begin{aligned}
\rm{I_{3}}&\lesssim\sum^{\infty}_{k=-\infty}\left(\sum^{\infty}_{j=k+2}2^{j\alpha q_{1}}\|f\chi_{j}\|^{q_{1}}_{L^{p_{1}(\cdot)}(\mathbb{R}^{n})}2^{(k-j)(n\delta_{1}+\alpha)q_{1}/2}\right)^{\frac{q_{1}}{q'_{1}}}
\left(\sum^{\infty}_{j=k+2}2^{(k-j)(n\delta_{1}+\alpha)q'_{1}/2}\right)^{\frac{q_{1}}{q'_{1}}}\\
&\lesssim\sum^{\infty}_{k=-\infty}\sum^{\infty}_{j=k+2}2^{j\alpha q_{1}}\|f\chi_{j}\|^{q_{1}}_{L^{p_{1}(\cdot)}(\mathbb{R}^{n})}2^{(k-j)(n\delta_{1}+\alpha)q_{1}/2}\\
&\lesssim\sum^{\infty}_{j=-\infty}2^{j\alpha q_{1}}\|f\chi_{j}\|^{q_{1}}_{L^{p_{1}(\cdot)}(\mathbb{R}^{n})}\sum^{j+2}_{k=-\infty}2^{(k-j)(n\delta_{1}+\alpha)q_{1}/2}\\
&\lesssim\|f\|^{q_{1}}_{\dot{K}^{\alpha,q_{1}}_{p_{1}(\cdot)}(\mathbb{R}^{n})}.\\
\end{aligned}
\end{equation*}

Summarizing the  estimates of ${\rm I_{1}}$ and ${\rm I_{3}}$, we conclude that (i) holds.  The proof of (ii) is similar to (i). The details being omitted. We finish the proof of Theorem \ref{T-HH}.
\end{proof}

From the proof in\cite[pp.\,71-72]{RA} and \cite[Lemma 3.1]{FC1}, it follows that
$$\left\|\sup_{Q}\frac{1}{|Q|^{1+\gamma/n}}\int_{Q}|h^{Q}|\right\|_{L^{q(\cdot)}(\mathbb{R}^{n})}\leq C\left\|\sup_{Q}\frac{1}{|Q|^{1+\gamma/n+\alpha/n}}\int_{Q}|h^{Q}|\right\|_{L^{p(\cdot)}(\mathbb{R}^{n})}.$$
By the above fact, the following lemma follows from the definitions of $\dot{K}^{\alpha,q}_{p(\cdot)}(\mathbb{R}^{n})$ and $K^{\alpha,q}_{p(\cdot)}(\mathbb{R}^{n})$.

\begin{lem}\label{L5}
Let $p_{1}(\cdot),p_{2}(\cdot)\in\mathscr{B}(\mathbb{R}^{n})$ and $\frac{1}{p_{1}(\cdot)}-\frac{1}{p_{2}(\cdot)}=\frac{\lambda}{n}$.
Suppose that function $h^{Q}$ is  defined on the cube $Q$. Then for $0\leq\beta$,
\begin{equation*}
\left\|\sup_{Q}\frac{1}{|Q|^{1+\beta/n}}\int_{Q}|h^{Q}|\right\|_{\dot{K}^{\alpha, q}_{p_{2}(\cdot)}(\mathbb{R}^{n})}\leq C\left\|\sup_{Q}\frac{1}{|Q|^{1+\beta/n+\lambda/n}}\int_{Q}|h^{Q}|\right\|_{\dot{K}^{\alpha,q}_{p_{1}(\cdot)}(\mathbb{R}^{n})}
\end{equation*}
and
\begin{equation*}
\left\|\sup_{Q}\frac{1}{|Q|^{1+\beta/n}}\int_{Q}|h^{Q}|\right\|_{K^{\alpha,q}_{p_{2}(\cdot)}(\mathbb{R}^{n})}\leq C\left\|\sup_{Q}\frac{1}{|Q|^{1+\beta/n+\lambda/n}}\int_{Q}|h^{Q}|\right\|_{K^{\alpha,q}_{p_{1}(\cdot)}(\mathbb{R}^{n})},
\end{equation*}
where the constant C depends only on $p_{1}(\cdot),p_{2}(\cdot),q,\alpha$ and n.
\end{lem}

\begin{proof}[\bf{Proof of Theorem \ref{H2}}]
First, we verify (i) of Theorem \ref{H2}.
Fix a cube $Q=Q(c_{Q},t)$ with $c_{Q}\in\rn$ and $t>0$. Let
$3^{k}Q=Q(c_{Q}, 3^{k}t)$ and $z\in3^{k}Q\backslash(3^{k-1}Q)$ for $k\geq2$.
For $g\in L^{p_{1}(\cdot)}(\mathbb{R}^{n}),$ let $g_{1}=g\chi_{3Q}$ and $g_{2}=g-g_{1}$.
From \eqref{eq-xQ}, it follows that there exists $x_{Q}\in Q$ such that
\begin{align}\label{C-xQ-C}
C_{b, g, Q}(x_{Q})
:=3\int_{\rn\backslash(3Q)}\frac{|b(z)-b_{3Q}||g(z)|}{|x_{Q}-z|^{n-\lambda}}dz
\leq\sup_{y\in3Q}\int_{\rn\backslash(3Q)}\frac{|b(z)-b_{3Q}||g(z)|}{|y-z|^{n-\lambda}}dz
=:C_{b, g, Q}.
\end{align}
Using \eqref{eq-H3-dot}, we repeat the process of \eqref{eq-f123} and obtain
\begin{equation*}
\begin{aligned}
&\|[b,T_{\lambda}](g)\|_{\dot{K}^{\alpha,q_{2}}_{p_{2}(\cdot)}\dot{F}^{\beta}_{\infty}}\\
&\sim\left\|\sup_{Q}\frac{1}{|Q|^{1+\beta/n}}\int_{Q}|[b,T_{\lambda}](g)(y)-([b,T_{\lambda}](g))_{Q}|dy\right\|_{\dot{K}^{\alpha,q_{2}}_{p_{2}(\cdot)}}\\
&\lesssim\left\|\sup_{Q}\frac{1}{|Q|^{1+\beta/n}}\int_{Q}|[b, T_{\lambda}](g)(y)
-C_{b, g, Q}(x_{Q})|dy\right\|_{\dot{K}^{\alpha,q_{2}}_{p_{2}(\cdot)}}\\
&\lesssim\left\|\sup_{Q}\frac{1}{|Q|^{1+\beta/n}}\int_{3Q}|(b(y)-b_{3Q})Tg(y)|dy\right\|_{\dot{K}^{\alpha,q_{2}}_{p_{2}(\cdot)}}\\
&\qquad+\left\|\sup_{Q}\frac{1}{|Q|^{1+\beta/n}}\int_{3Q}|b_{3Q}Tg(y)-T(gf)(y)-C_{b, g, Q}(x_{Q})|dy\right\|_{\dot{K}^{\alpha,q_{2}}_{p_{2}(\cdot)}}\\
&=:\rm{J_{1}+J_{2}},
\end{aligned}
\end{equation*}

To estimate $\rm{J_{1}}$, by Lemma \ref{L3}, we deduce that for all $x\in Q$,
\begin{align*}
\frac{1}{|Q|^{1+\beta/n}}\int_{Q}|(b(y)-b_{Q})T_{\lambda}(g)(y)|dy
&\lesssim\frac{1}{|Q|^{\beta/n}}\sup_{y\in Q}|(b(y)-b_{Q})|
\left(\frac{1}{|Q|}\int_{Q}|T_{\lambda}(g)(y)|dy\right)\\
&\lesssim\|b\|_{\dot{\Lambda}_{\beta}}M(T_{\lambda}(g))(x).
\end{align*}
This and Theorem \ref{T-HH} imply that
$${\rm{J}}_{1}\lesssim\|b\|_{\dot{\Lambda}_{\beta}}\|T_{\lambda}(g)\|_{\dot{K}^{\alpha,q_{2}}_{p_{2}(\cdot)}(\mathbb{R}^{n})}
\lesssim\|b\|_{\dot{\Lambda}_{\beta}}\|g\|_{\dot{K}^{\alpha,q_{1}}_{p_{1}(\cdot)}(\mathbb{R}^{n})}.$$

Now, we consider ${\rm{J}}_{2}$. According to \eqref{1-F} and \eqref{2-F} in Theorem \ref{H1}, by Theorem \ref{T-HH}, we deduce
\begin{align*}
{\rm{J}}_{2}
&\lesssim\left\|\sup_{Q}\frac{1}{|Q|^{1+\beta/n}}
\int_{Q}|T_{\lambda}((b-b_{Q})g_{1})(y)|dy\right\|_{\dot{K}^{\alpha,q_{2}}_{p_{2}(\cdot)}}
+\|b\|_{\dot{\Lambda}_{\beta}}\|T_{\lambda}(g)\|_{\dot{K}^{\alpha,q_{2}}_{p_{2}(\cdot)}(\mathbb{R}^{n})}\\
&\qquad+\left\|\sup_{Q}\frac{1}{|Q|^{\beta/n}}
3C_{b, g, Q}-C_{b, g, Q}(x_{Q})\right\|_{\dot{K}^{\alpha,q_{2}}_{p_{2}(\cdot)}}\\
&\lesssim\left\|\sup_{Q}\frac{1}{|Q|^{1+\beta/n+\lambda/n}}
\int_{Q}|T_{\lambda}((b-b_{Q})g_{1})(y)|dy\right\|_{\dot{K}^{\alpha,q_{2}}_{p_{1}(\cdot)}}\\
&\qquad+\left\|\sup_{Q}\frac{1}{|Q|^{\beta/n+\lambda/n}}
3C_{b, g, Q}-C_{b, g, Q}(x_{Q})\right\|_{\dot{K}^{\alpha,q_{2}}_{p_{1}(\cdot)}}\\
&\qquad+\|b\|_{\dot{\Lambda}_{\beta}}\|g\|_{\dot{K}^{\alpha,q_{1}}_{p_{1}(\cdot)}(\mathbb{R}^{n})}\\
&=:{\rm{J}}_{21}+{\rm{J}}_{22}+\|b\|_{\dot{\Lambda}_{\beta}}\|g\|_{\dot{K}^{\alpha,q_{1}}_{p_{1}(\cdot)}(\mathbb{R}^{n})}.
\end{align*}

Below, we estimate ${\rm{J}}_{21}$. Chooser $r$ and $\bar{r}$ such that $1<r<p_{1}^-$ and $(1/r-1/\bar{r})=(\lambda/n)$. Such $\bar{r}$~exists, since $r<p_{1}^-<n/\lambda$, it follows that
\begin{equation*}
\begin{aligned}
&\frac{1}{|Q|^{1+\beta/n+\lambda/n}}\int_{Q}|T_{\lambda}((b-b_{Q})g_{1})(y)|dy\\
&\lesssim\frac{1}{|Q|^{1+\beta/n+\lambda/n}}\|T_{\lambda}((b-b_{Q})g_{1})\|_{\bar{r}}|Q|^{1/\bar{r}'}\\
&\lesssim|Q|^{-1-(\lambda+\beta)/n+1-1/\bar{r}}\|(b-b_{Q})g\|_{r}\\
&\lesssim\|b\|_{\dot{\Lambda}_{\beta}}(M(|g|^{r}))^{1/r}.
\end{aligned}
\end{equation*}
So, by the boundedness of $M$ from $\dot{K}^{\alpha, q_{2}}_{p_{1}(\cdot)}(\mathbb{R}^{n})$ to $\dot{K}^{\alpha, q_{2}}_{p_{1}(\cdot)}(\mathbb{R}^{n})$,
we obtain ${\rm{J_{21}}}\lesssim\|b\|_{\dot{\Lambda}_{\beta}}\|g\|_{\dot{K}^{\alpha,q_{2}}_{p_{1}(\cdot)}(\mathbb{R}^{n})}$.
Notice that if $q_{1}\leq q_{2}$, then
\begin{equation*}
\dot{K}^{\alpha,q_{1}}_{p(\cdot)}(\mathbb{R}^{n})\subseteq \dot{K}^{\alpha,q_{2}}_{p(\cdot)}(\mathbb{R}^{n}).
\end{equation*}
This implies
$${\rm{J}}_{21}\lesssim\|b\|_{\dot{\Lambda}_{\beta}}\|g\|_{\dot{K}^{\alpha, q_{1}}_{p_{1}(\cdot)}(\mathbb{R}^{n})}.$$

To show $\rm{J_{22}}$, we repeat this derivation of $\rm{F_{22}}$ in Theorem \ref{H1} and obtain
\begin{equation*}
\begin{aligned}
\sup_{Q}\frac{1}{|Q|^{\beta/n+\lambda/n}}3C_{b, g, Q}-C_{b, g, Q}(x_{Q})
\lesssim\|b\|_{\dot{\Lambda}_{\beta}}M(g)(x).
\end{aligned}
\end{equation*}
Similar to the finally estimate of $\rm{J_{21}}$, we also have $${\rm{J_{22}}}\lesssim\|b\|_{\dot{\Lambda}_{\beta}}\|g\|_{\dot{K}^{\alpha,q_{1}}_{p_{1}(\cdot)}(\mathbb{R}^{n})}.$$

Summarizing the  estimates of ${\rm J_{1}}$ and ${\rm J_{2}}$, we conclude that (i) of Theorem \ref{H2} holds.
To verify (ii) of Theorem \ref{H2}, by \eqref{eq-H3} and
\begin{equation*}
K^{\alpha,q_{1}}_{p(\cdot)}(\mathbb{R}^{n})\subseteq K^{\alpha,q_{2}}_{p(\cdot)}(\mathbb{R}^{n})\quad \text{for all} \ q_{1}\leq q_{2},
\end{equation*}
we repeat the process of proof of (i) and obtain (ii).
This ends the proof of Theorem \ref{H2}.
\end{proof}

\section{Some classical operators}
\quad
It is worth that we assumed the boundedness of sublinear operators on the variable exponent Lebesgue space in the discussion of Theorem \ref{H1} and Theorem \ref{H2}. In order to ensure that some classical operators satisfy this fact, we introduce the following  $\log$-H\"{o}lder continuity condition.

\begin{defn}\label{Hod}
\begin{enumerate}
\item[\rm (i)]
A function $p: \mathbb{R}^{n}\to\mathbb{R}$  is locally $\log$-H\"{o}lder continuous on $\mathbb{R}^{n}$ if  there exists $c_{1}>0$ such that
\begin{equation*}
|p(x)-p(y)|\leq\frac{c_{1}}{\log(e+1/|x-y|)} \quad \text{for all} \ x, y\in\mathbb{R}^{n}.
\end{equation*}

\item[\rm (ii)]
A function $p: \mathbb{R}^{n}\to\mathbb{R}$  is $\log$-H\"{o}lder continuous at the origin (or has a $\log$ decay at the origin) if there exists $c_{2}>0$ such that
\begin{equation*}
|p(x)-p(0)|\leq\frac{c_{2}}{\log(e+1/|x|)} \quad \text{for all} \ x\in\mathbb{R}^{n}.
\end{equation*}
The notation $\mathscr{P}^{\log}_{0}(\mathbb{R}^{n})$ is used for all exponents $p(\cdot)\in\mathscr{P}(\mathbb{R}^{n})$ which is $\log$-H\"{o}lder continuous at the origin.

\item[\rm (iii)]
A function $p: \mathbb{R}^{n}\to\mathbb{R}$  is $\log$-H\"{o}lder continuous  at infinity (or has a $\log$ decay  at infinity)
if there exists $c_{3}>0$ such that
\begin{equation*}
|p(x)-p_{\infty}|\leq\frac{c_{3}}{\log(e+|x|)} \quad \text{for all} \ x\in\mathbb{R}^{n},
\end{equation*}
where $p_{\infty}=\lim_{|x|\to\infty}p(x)$.
The notation $\mathscr{P}^{\log}_{\infty}(\mathbb{R}^{n})$ is used for all exponents $p(\cdot)\in\mathscr{P}(\mathbb{R}^{n})$
which is $\log$-H\"{o}lder continuous at infinity.
By $\mathscr{P}^{\log}(\mathbb{R}^{n})$ we define the set of all exponents $p(\cdot)\in\mathscr{P}(\mathbb{R}^{n})$ which is locally $\log$-H\"{o}lder continuous on $\mathbb{R}^{n}$ and has a $\log$ decay  at infinity.
\end{enumerate}
\end{defn}

Obviously, we have $\mathscr{P}^{\log}(\mathbb{R}^{n})\subset(\mathscr{P}^{\log}_{0}(\mathbb{R}^{n})\cap\mathscr{P}^{\log}_{\infty}(\mathbb{R}^{n}))$.
Note that $p(\cdot)\in\mathscr{P}^{\log}(\mathbb{R}^{n})$  if and only if $p'(\cdot)\in\mathscr{P}^{\log}(\mathbb{R}^{n})$,
where $1/p'(\cdot)=1-1/p(\cdot)$.
If $p(\cdot)\in\mathscr{P}^{\log}(\mathbb{R}^{n})$, then the Hardy-Littlewood maximal operator $M$ is bounded on $L^{p(\cdot)}(\mathbb{R}^{n})$
(see \cite[Theorem 4.3.8]{DHHR}), which further implies $p(\cdot)\in\mathscr{B}(\mathbb{R}^{n})$.

Given a locally integrable function $K$ defined on $\mathbb{R}^{n}\backslash\{0\}$, suppose that the Fourier transform of $K$ is bounded,
$$
|K(x)|\leq\frac{C}{|x|^{n}}\quad \text{and} \quad |\nabla K(x)|\leq\frac{C}{|x|^{n+1}} \quad \text{for all} \ x\neq0.
$$
Then the classical Calder\'{o}n-Zygmund singular integral operator $T$ is defined by
\begin{align}\label{def-T}
Tf(x)=K\ast f(x)\quad \text{for all} \ x\in\rn.
\end{align}
Based on the above definition, the following lemma comes from \cite[Corollary 2.5]{CFMP}.
\begin{lem}\label{lem-T-b}
Let $p(\cdot)\in\mathscr{B}(\mathbb{R}^{n})$ and $T$ be as in \eqref{def-T}. Then $T$ is bounded on $L^{p(\cdot)}(\mathbb{R}^{n})$.
\end{lem}

By the condition $|K(x)|\leq C/(|x|^{n})$ for $x\neq0,$ we deduce that $T$ satisfies \eqref{sub-def} and \eqref{size-1}.
In addition, it is obvious that the Hardy-Littlewood maximal operator $M$ satisfies \eqref{sub-def} and \eqref{size-1}.
Thus, the following corollary follows from Theorem \ref{H1} and Lemma \ref{lem-T-b}.
\begin{cor}
Let $\alpha\in \mathbb{R}$, $0<\beta<1$, $0<q\leq\infty$, $\delta_{1}, \delta_{2}\in(0,1)$ satisfy (\ref{1.1}) and $-n\delta_{1}<\alpha<n\delta_{2}$.
Assume that $b\in \dot{\Lambda}_{\beta}$ and $T$ is as in \eqref{def-T}. Then the following assertions hold:
\begin{enumerate}
\item[\rm (i)]
If $p(\cdot)\in\mathscr{B}(\mathbb{R}^{n})$, then $[b, T]$ and $[b, M]$ are bounded from $\dot{K}^{\alpha, q}_{p(\cdot)}(\mathbb{R}^{n})$ to $\dot{K}^{\alpha,q}_{p(\cdot)}\dot{F}^{\beta}_{\infty}(\mathbb{R}^{n})$.
\item[\rm (ii)]
If $p(\cdot)\in\mathscr{B}(\mathbb{R}^{n})$, then $[b, T]$ and $[b, M]$ are bounded from $K^{\alpha, q}_{p(\cdot)}(\mathbb{R}^{n})$ to $K^{\alpha,q}_{p(\cdot)}\dot{F}^{\beta}_{\infty}(\mathbb{R}^{n})$.
\end{enumerate}
\end{cor}

For any $0<\lambda<n$ and $x\in\rn$,  the Riesz potential operator $I_{\lambda}$ is defined by
\begin{align*}
I_{\lambda}f(x)=\int_{\rn}\frac{|f(y)|}{|x-y|^{n-\lambda}}dy.
\end{align*}
Then the Riesz potential operator $I_{\lambda}$ satisfies \eqref{sub-def} and \eqref{size-2}.
The following lemma about the boundedness of $I_{\lambda}$ on the variable exponent Lebesgue space comes from \cite[Theorem 6.1.9]{DHHR}.
\begin{lem}\label{lem-I-b}
Let $0<\lambda<n$, $p_{1}(\cdot)\in\mathscr{P}^{\log}(\mathbb{R}^{n})$ and $1/p_{2}(\cdot)=1/p_{1}(\cdot)-\lambda/n$.
If $1<p_{1}^{-}\leq p_{1}^{+}<n/\lambda$, then $I_{\lambda}$ is bounded from $L^{p_{1}(\cdot)}(\mathbb{R}^{n})$ to $L^{p_{2}(\cdot)}(\mathbb{R}^{n})$.
\end{lem}

Correspondingly, for any $0<\lambda<n$, the fractional maximal function $M_{\lambda}$ is defined by
\begin{align*}
M_{\lambda}f(x):=\sup_{r>0}r^{\lambda-n}\int_{B(x,r)}|f(y)|dy \quad \text{for all} \ x\in\rn.
\end{align*}
Then the fractional maximal function $M_{\lambda}$ also satisfies \eqref{sub-def} and \eqref{size-2}. Moreover, we have
$$M_{\lambda}(f)\ls I_{\lambda}(|f|).$$
This, along with Theorem \ref{H2} and Lemma \ref{lem-I-b}, yields the following corollary.

\begin{cor}
Let $\alpha\in \mathbb{R}$, $0<q_{1},q_{2}<\infty$, $1<p_{1}^{-}\leq p_{1}^{+}<n/\lambda$ and $1/p_{2}(\cdot)=1/p_{1}(\cdot)-\lambda/n$. Suppose that $\delta_{1}, \delta_{2}\in(0,1)$ satisfy (\ref{1.1}) and $-n\delta_{1}<\alpha<n\delta_{2}$.
Assume further that $b\in \dot{\Lambda}_{\beta}$ with $0<\beta<1$. Then the following assertions hold:
\begin{enumerate}
\item[\rm (i)]
If $p_{1}(\cdot)\in\mathscr{P}^{\log}(\mathbb{R}^{n})$, then $[b, I_{\lambda}]$ and $[b, M_{\lambda}]$ are bounded
from $\dot{K}^{\alpha,q_{1}}_{p_{1}(\cdot)}(\mathbb{R}^{n})$
to $\dot{K}^{\alpha,q_{2}}_{p_{2}(\cdot)}\dot{F}^{\beta}_{\infty}(\mathbb{R}^{n})$.
\item[\rm (ii)]
If $p_{1}(\cdot)\in\mathscr{P}^{\log}(\mathbb{R}^{n})$, then $[b, I_{\lambda}]$ and $[b, M_{\lambda}]$ are bounded
from $K^{\alpha,q_{1}}_{p_{1}(\cdot)}(\mathbb{R}^{n})$
to $K^{\alpha,q_{2}}_{p_{2}(\cdot)}\dot{F}^{\beta}_{\infty}(\mathbb{R}^{n})$.
\end{enumerate}
\end{cor}

{\bf Acknowledgments:}
Supported by the National Natural Science Foundation of xinjiang Province of China (No:2021D01C463).

Chenglong Fang

School of Mathematics, Renmin University of China, Beijing 100872, China

\smallskip

{\it E-mail}: \texttt{fangclmath@126.com }

\vspace{0.3cm}

Yingying Wei

School of Mathematics and Statistics, Yili Normal University, Yining Xinjiang 835000, China

\smallskip

{\it E-mail}: \texttt{wyy552@126.com }

\vspace{0.3cm}

Jing Zhang

School of Mathematics and Statistics, Yili Normal University, Yining Xinjiang 835000, China

\smallskip

{\it E-mail}: \texttt{zjmath66@126.com}

\end{document}